\documentclass{amsart}

\usepackage{color}
\usepackage{graphicx,epsf}
\usepackage [cmtip,arrow]{xy}
\usepackage {pb-diagram,pb-xy} 

\newtheorem{theorem}{Theorem}[section]
\newtheorem{lemma}[theorem]{Lemma}
\newtheorem{corollary}[theorem]{Corollary}
\newtheorem{proposition}[theorem]{Proposition}
\theoremstyle{definition}
\newtheorem{definition}[theorem]{Definition}

\theoremstyle{remark}
\newtheorem{remark}[theorem]{Remark}
\newtheorem{observation}[theorem]{Observation}
\newtheorem{notation}[theorem]{Notation}
\numberwithin{equation}{section}

\include{diagrams}

\newcommand{\Z}{\mathbb Z}

\newcommand{\Rc}{{\rm  R}}
\newcommand{\R}{{\mathbb  R}}

\newcommand{\hocolimit}{{\mathrm{hocolim}}}

\newcommand{\eps}{\varepsilon}

\newcommand{\x}{\mathbf{x}}

\newcommand{\y}{\mathbf{y}}

\newcommand{\z}{\mathbf{z}}

\newcommand{\A} {\mathcal{A}}
\newcommand{\B} {\mathcal{B}}
\newcommand{\HH} {{\rm H}}

\newcommand {\hide}[1]{}

\renewcommand{\x}{\mathbf{x}}

\renewcommand{\y}{\mathbf{y}}

\newcommand{\s}{\mathbf{s}}

\renewcommand{\z}{\mathbf{z}}

\newcommand{\Suspension}{{\mathbf  S}}

\newcommand{\dist}{{\rm  dist}}

\begin{document}

\title[Topological types of parametrized arrangements]{
On the number of topological  types occurring in a parametrized family 
of arrangements
\footnote{2000 Mathematics Subject Classification 14P10, 14P25}}
\author{Saugata Basu}
\address{School of Mathematics,
Georgia Institute of Technology, Atlanta, GA 30332, U.S.A.}
\email{saugata.basu@math.gatech.edu}
\thanks{The author was supported in part by NSF grant CCF-0634907.}

\begin{abstract}
Let ${\mathcal S}(\R)$ be an o-minimal structure 
over $\R$,
$T \subset \R^{k_1+k_2+\ell}$ a closed definable set, 
and 
$$
\displaylines{
\pi_1: \R^{k_1+k_2+\ell}\rightarrow \R^{k_1 + k_2}, \\
\pi_2: \R^{k_1+k_2+\ell}\rightarrow \R^{\ell}, \\
\pi_3: \R^{k_1 + k_2} \rightarrow \R^{k_2} 
}
$$
the projection maps as depicted below.

\[
\begin{diagram}
\node{\R^{k_1 + k_2 + \ell}}\arrow{e,t}{\pi_1}
\arrow{s,l}{\pi_2} 
\node{\R^{k_1+k_2}}\arrow{s,l}{\pi_3}     \\
\node{\R^{\ell}}\node{\R^{k_2}}
\end{diagram}
\]
For any collection ${\mathcal A} = \{A_1,\ldots,A_n\}$ of 
subsets of $\R^{k_1+k_2}$,
and $\z \in \R^{k_2}$, let $\A_\z$ denote the collection of
subsets of $\R^{k_1}$ 
\[
\{A_{1,\z},\ldots, A_{n,\z}\}\]
where 
$A_{i,\z} = A_i \cap \pi_3^{-1}(\z), \;1 \leq i \leq n$.
We prove that there exists a constant $C = C(T) > 0$ such that 
for any family 
${\mathcal A} = \{A_1,\ldots,A_n\}$ of definable sets,
where each $A_i = \pi_1(T \cap \pi_2^{-1}(\y_i))$,
for some  $\y_i \in \R^{\ell}$,
the number of distinct stable homotopy types amongst the arrangements
$\A_\z,\; \z \in \R^{k_2}$ is bounded by
$
\displaystyle{
C \cdot n^{(k_1+1)k_2}
}
$
while the number of distinct homotopy types is bounded by
$
\displaystyle{
C \cdot n^{(k_1+3)k_2}.
}
$
This generalizes to the 
o-minimal setting, bounds of the same type
proved in \cite{BV06} for semi-algebraic and semi-Pfaffian families.
One  technical tool used in the proof of the above results
is a pair of topological comparison theorems reminiscent of Helly's theorem
in convexity theory and these  might be of independent
interest in the quantitative study of arrangements.
\end{abstract}
\maketitle

\section{Introduction}

The study of arrangements is a very important subject in discrete and
computational geometry, where one studies {\em arrangements} of $n$ 
subsets of $\R^k$ 
(often referred to as objects of the arrangements)
for fixed $k$ and large values of $n$ (see \cite{Agarwal} for a survey
of the known results from this area).
The precise nature of the objects in an arrangements 
will be discussed in more details below. Common examples 
consist of arrangements of hyperplanes, balls or simplices in $\R^k$.
More generally one considers  arrangements of objects of ``bounded
 description complexity''.
This means that each set in the arrangement is defined by a first order 
formula in the language of ordered fields involving 
at most a constant number of polynomials whose  degrees are also 
bounded by a constant (see \cite{Matousek}).

In this paper we consider parametrized families of 
arrangements. The question we will be interested in most, is the number
of ``topologically'' distinct arrangements which can occur in such a family
(precise definition of the topological type of an arrangement 
is given later (see Definition \ref{def:topologicaltype})).
Parametrized arrangements occur quite frequently in practice. For instance,
take any arrangement $\A$ in $\R^{k_1+k_2}$ and let 
$\pi:\R^{k_1+k_2} \rightarrow \R^{k_2}$ be the projection on the last
$k_2$ co-ordinates. Then for each $\z \in \R^{k_2}$, 
the intersection of the arrangement $\A$ with the  fiber $\pi^{-1}(\z)$, is
an arrangement $\A_\z$ in $\R^{k_1}$ and the family of the arrangements
$\{\A_\z\}_{\z \in \R^{k_2}}$ is an example of a parametrized family
of arrangements. Even though the number of arrangements in the family
$\{A_\z\}_{\z \in \R^{k_2}}$ is infinite, it follows from 
Hardt's triviality theorem  
generalized to  o-minimal structures (see Theorem \ref{the:hardt} below)
that the number of ``topological types'' occurring
amongst them is finite and can be effectively bounded in terms of the
$n,k_1,k_2$ 
up to multiplication by a constant that depends only on the 
particular  family from which the objects of the arrangements are drawn.
If by topological type we mean homeomorphism type, then the best known 
upper bound on the number of types occurring is doubly exponential in
$k_1,k_2$. However, if we consider the weaker notion of homotopy type, then
we obtain a singly exponential bound.
We conjecture that a singly exponential bound also holds for homeomorphism
types as well.

We now make precise the class of arrangements that we consider and also the
notion of topological type of an arrangement.

\subsection{Combinatorial Complexity in O-minimal Geometry}
In order to put the study of the combinatorial complexity of
arrangements in a more natural mathematical context, as well
as to elucidate the proofs of the main results in the area, 
a new framework was introduced in \cite{Basu9}
which is a significant generalization of the settings mentioned above.
We recall here the basic definitions of this framework
from \cite{Basu9}, referring the reader
to the same paper for further details and examples.

We first recall an important model theoretic notion --
that of o-minimality -- which plays a crucial role in this
generalization.
  
\subsubsection{O-minimal Structures}
O-minimal structures were invented and first studied by
Pillay and Steinhorn in the pioneering papers
\cite{PS1,PS2}. Later the theory was further
developed through contributions of other researchers, most notably
van den Dries, Wilkie, Rolin, Speissegger amongst others
\cite{Dries2,Dries3,Dries4,Wilkie,Wilkie2,Rolin}. We particularly
recommend the book by van den Dries \cite{Dries} and the notes by
Coste \cite{Michel2} for an easy introduction to the topic as well as the
proofs of the basic results that we use in this paper.

\begin{definition}[o-minimal structure]
\label{def:o-minimal}
An o-minimal structure over 
a real closed field 
$\Rc$ is a sequence 
${\mathcal S}(\Rc) = ({\mathcal S}_n)_{n \in {\mathbb N}}$, where 
each ${\mathcal S}_n$ is a collection of subsets of $\Rc^n$
(called the {\em definable sets} in the structure) satisfying the 
following axioms (following the exposition in \cite{Michel2}). 

\begin{enumerate}
\item
All algebraic subsets of $\Rc^n$ are in ${\mathcal S}_n$.
\item
The class ${\mathcal S}_n$ is closed under complementation and
finite unions and intersections.
\item
If $A \in {\mathcal S}_m$ and $B \in {\mathcal S}_n$ then
$A \times B \in {\mathcal S}_{m+n}$.
\item
If $\pi: \Rc^{n+1} \rightarrow \Rc^{n}$ is the projection map on the
first $n$ co-ordinates and $A \in {\mathcal S}_{n+1}$, then 
$\pi(A)  \in {\mathcal S}_n$.
\item
The elements of ${\mathcal S}_1$ are precisely finite unions of points
and intervals.
\end{enumerate}
\end{definition}

The class of semi-algebraic sets is one obvious example of such a structure,
but in fact there are much richer classes of sets which have been proved
to be o-minimal (see \cite{Michel2,Dries}). 

\subsubsection{Admissible Sets}
We now recall from \cite{Basu9} the definition of
the class of sets that will play the role of sets with
bounded description complexity mentioned above.

\begin{definition}[admissible sets]
\label{def:admissible}
Let ${\mathcal S}(\R)$ be an o-minimal structure over 
$\R$ and let $T \subset \R^{k+\ell}$
be a fixed definable set. 
Let $\pi_1: \R^{k+\ell} \rightarrow \R^{k}$
(respectively $\pi_2: \R^{k+\ell}  \rightarrow \R^{\ell}$) 
be the projections onto the first $k$ (respectively last $\ell$) co-ordinates.

\[
\begin{diagram}
\node{}\node{T \subset \R^{k+\ell}}
\arrow{sw,t}{\pi_1}
\arrow{se,t}{\pi_2}
\node{}\\
\node{\R^{k}}\node{}\node{\R^{\ell}}
\end{diagram}
\]

We will call a subset $S$ of $\R^k$ to be a $(T,\pi_1,\pi_2)$-set if
\[
S = T_\y = \pi_1(\pi_2^{-1}(\y)\cap T)
\]
for some $\y \in \R^{\ell}$. 

If $T$ is some fixed  definable set, we
call a family of $(T,\pi_1,\pi_2)$-sets to be a 
$(T,\pi_1,\pi_2)$-family. We wil also refer to a finite
$(T,\pi_1,\pi_2)$-family as an {\em arrangement} of $(T,\pi_1,\pi_2)$-sets.
\end{definition}

\subsection{Stable Homotopy Equivalence}
\label{subsec:stable}
For any finite 
CW-complex $X$ we denote by $\Suspension X$ the suspension of 
$X$ and for $n \geq 0$, we denote by  
$\Suspension^n X$ the $n$-fold iterated suspension
$\underbrace{\Suspension \circ \Suspension \circ \cdots \circ
\Suspension}_{n \text{ times}}X$. 

Note that if $i: X \hookrightarrow Y$ is an inclusion map,
then there is an obvious induced inclusion map
$\Suspension^n i: \Suspension^n X \hookrightarrow \Suspension^n Y$ between
the $n$-fold iterated suspensions of $X$ and $Y$.

Recall from \cite{Spanier-Whitehead}
that for two finite CW-complexes $X$ and $Y$, an element of
\begin{equation}
\label{eqn:defofS-maps}
\{X;Y\}= \varinjlim_i \; [\Suspension^i X,\Suspension^i Y]
\end{equation}
is called an {\em S-map} (or map in the {\em suspension category}).
An S-map $f \in \{X;Y\}$ is represented by the homotopy class of a map
$f: \Suspension^N X \rightarrow \Suspension^N Y$ for some $N \geq 0$.

\begin{definition}[stable homotopy equivalence]
\label{def:S-equivalence}
An S-map $f \in \{X;Y\}$ is an S-equivalence 
(also called a stable homotopy equivalence) if it admits an
inverse $f^{-1} \in \{Y;X\}$. In this case we say that $X$ and $Y$ are
stable homotopy equivalent.
\end{definition}

If $f \in \{X;Y\}$ is an S-map, then $f$ induces a homomorphism
\[
f_* : \HH_*(X,\Z) \rightarrow \HH_*(Y,\Z)
\]
between the homology groups of $X$ and $Y$. 

The following theorem characterizes stable homotopy equivalence in terms of
homology.

\begin{theorem}
\label{the:stable}
\cite[pp. 604]{Dieudonne}
Let $X$ and $Y$ be two finite CW-complexes. 
Then $X$ and $Y$ are stable homotopy
equivalent if and only if  there exists an S-map
$f \in \{X;Y\}$ 
which induces isomorphisms $f_* : \HH_*(X,\Z) \rightarrow \HH_*(Y,\Z)$.
\end{theorem}

\subsection{Diagrams and Co-limits}
The arrangements that we consider are all finitely triangulable.
In other words, the union of objects of an arrangement
is homeomorphic to a finite simplicial complex,
and each individual  object in the arrangement will correspond to a 
sub-complex of this simplicial complex. 
It will be more convenient to work in the category 
of finite regular cell complexes, instead of just simplicial complexes.

Let ${\mathcal A} = \{A_1,\ldots,A_n\}$,  where each $A_i$ is a sub-complex
of a finite regular cell complex.
We will denote by $[n]$ the set $\{1,\ldots,n\}$ and
for $I \subset [n]$ we will denote by 
${\mathcal A}^I$ (respectively ${\mathcal A}_I$) the regular cell complexes
$\displaystyle{\bigcup_{i \in I} A_i}$ (respectively 
$\displaystyle{\bigcap_{i \in I} A_i}$). Notice that if
$J \subset I \subset [n]$, then
$$
\displaylines{
{\mathcal A}^J \subset {\mathcal A}^I, \cr
{\mathcal A}_I \subset {\mathcal A}_J .
}
$$

We will call the collection of sets $\{|\A_I|\}_{I \subset [n]}$ together with
the inclusion maps $i_{I,J}: |\A_I| \hookrightarrow |\A_J|, J \subset I$, the
{\em diagram} of $\A$.
Notice that 
(even though we do not use this fact), 
$|\A^{[n]}|$ is the co-limit of the diagram of $\A$.
For $I \subset [n]$ we will denote by $\A[I]$ the sub-arrangement
$\{ A_i \,\mid \, i \in I\}$.

\subsection{Diagram Preserving Maps}
Now let 
$
{\mathcal A} = \{A_1,\ldots,A_n\}, \;
{\mathcal B} = \{B_1,\ldots,B_n\}
$ 
where each $A_i,B_j$ is a sub-complex of a finite regular cell complex 
for $1 \leq i,j \leq n$.

\begin{definition}[diagram preserving maps]
\label{def:maps}
We call a map $f: |\A^{[n]}| \rightarrow |\B^{[n]}|$ 
to be {\em diagram preserving}
if $f(|\A_I|) \subset |\B_I|$ for every $I \subset [n]$. 
(Notice that the above property is equivalent to 
$f(|A_i|) \subset |B_i|$ for every $i \in [n]$ but the previous property
will be more convenient for us later when we extend the definition    
of diagram preserving maps to homotopy co-limits (see 
Definition \ref{def:dgpheofhocolimits}).)
We say that two maps
$f,g: |\A^{[n]}| \rightarrow |\B^{[n]}|$ are {\em diagram homotopic} if 
there exists a homotopy 
$h: |\A^{[n]}| \times [0,1] \rightarrow |\B^{[n]}|$, such that
$h(\cdot,0)= f, h(\cdot,1) = g$ and $h(\cdot,t)$ is diagram preserving for
each $t \in [0,1]$.

More generally,
we call a map 
$f: \Suspension^N |\A^{[n]}| \rightarrow \Suspension^N |\B^{[n]}|$ 
to be {\em diagram preserving}
if $f(\Suspension^N |\A_I|) \subset \Suspension^N |\B_I|$ 
for every $I \subset [n]$. We say that two maps
$f,g: \Suspension^N |\A^{[n]}| \rightarrow \Suspension^N |\B^{[n]}|$ 
are {\em diagram homotopic} if there exists
a homotopy $h: \Suspension^N |\A^{[n]}| \times [0,1] \rightarrow 
\Suspension^N |\B^{[n]}|$ such that
$h(\cdot,0)= f, h(\cdot,1) = g$ and $h(\cdot,t)$ is diagram preserving for
each $t \in [0,1]$.

We say that $f: | \A^{[n]}| \rightarrow |\B^{[n]}|$ is a diagram preserving
homeomorphism if there exists a diagram preserving  inverse map
$g: |\B^{[n]}| \rightarrow |\A^{[n]}|$  such that the induced maps
$g\circ f : | \A^{[n]}| \rightarrow | \A^{[n]}| $ and
$f\circ g : |\B^{[n]}| \rightarrow |\B^{[n]}|$ are 
${\mathrm{Id}}_{| \A^{[n]}|}$ and 
${\mathrm{Id}}_{|\B^{[n]}|}$, respectively.

We say that $f: |\A^{[n]}| \rightarrow |\B^{[n]}|$ is a diagram preserving
homotopy equivalence if there exists a diagram preserving  inverse map
$g: |\B^{[n]}| \rightarrow | \A^{[n]}|$  such that the induced maps
$g\circ f : | \A^{[n]}| \rightarrow | \A^{[n]}|$ and
$f\circ g : |\B^{[n]}| \rightarrow |\B^{[n]}|$ are diagram
homotopic to ${\mathrm{Id}}_{| \A^{[n]}|}$ and 
${\mathrm{Id}}_{|\B^{[n]}|}$, respectively.

We say that an S-map $f \in \{ |\A^{[n]}|;|\B^{[n]}| \}$ 
is a diagram preserving stable homotopy equivalence if it is
represented by a diagram preserving map
\[
\tilde{f}: \Suspension^N | \A^{[n]}| \rightarrow \Suspension^N |\B^{[n]}|
\]
such that there exists a diagram preserving  inverse map
\[
\tilde{g}: \Suspension^N |\B^{[n]}| \rightarrow \Suspension^N | \A^{[n]}|
\]
for which the induced maps
\[
\tilde{g}\circ \tilde{f} : \Suspension^N | \A^{[n]}| \rightarrow \Suspension^N | \A^{[n]}|,
\] 
and
\[
\tilde{f}\circ \tilde{g} : \Suspension^N |\B^{[n]}| \rightarrow \Suspension^N |\B^{[n]}|
\]
are diagram homotopic to ${\mathrm{Id}}_{\Suspension^N | \A^{[n]}|}$ and 
${\mathrm{Id}}_{\Suspension^N |\B^{[n]}|}$, respectively.
\end{definition}

Translating these topological definitions into the language of arrangements,
we say that:

\begin{definition}[topological type of an arrangement]
\label{def:topologicaltype}
Two arrangements $\A,\B$ are homeomorphic (respectively homotopy
equivalent, stable homotopy equivalent) if there exists a diagram preserving
homeomorphism (respectively homotopy equivalence, stable homotopy equivalence)
between them. 
\end{definition}

\begin{remark}
\label{rem:distinct}
Note that, since two definable sets might be stable homotopy equivalent,
without being homotopy equivalent (see \cite[pp. 462]{Spanier}), and also 
homotopy equivalent without being homeomorphic, the notions
of homeomorphism type, homotopy type and stable homotopy type
are each strictly weaker than the previous one.
\end{remark}

The main results of this paper can now be stated.

\subsection{Main Results}
\label{sec:application}

Let ${\mathcal S}(\R)$ be an o-minimal structure 
over 
$\R$,
$T \subset \R^{k_1+k_2+\ell}$ a closed and bounded definable set,
and let $\pi_1: \R^{k_1+k_2+\ell}\rightarrow \R^{k_1 + k_2}$
(respectively,
$\pi_2: \R^{k_1+k_2+\ell}\rightarrow \R^{\ell}$, 
$\pi_3: \R^{k_1 + k_2} \rightarrow \R^{k_2}$)
denote the projections onto the first $k_1 + k_2$ 
(respectively, the last $\ell$, the last $k_2$) co-ordinates. 
For any collection
${\mathcal A} = \{A_1,\ldots,A_n\}$ of $(T,\pi_1,\pi_2)$-sets,
and $\z \in \R^{k_2}$, we will denote by $\A_\z$ the collection of
sets, $\{A_{1,\z},\ldots, A_{n,\z}\}$, where 
$A_{i,\z} = A_i \cap \pi_3^{-1}(\z), 1 \leq i \leq n$.

A fundamental theorem in o-minimal geometry is Hardt's trivialization 
theorem (Theorem \ref{the:hardt} below) which says  that there exists a
definable partition of $\R^{k_2}$ into a finite number of 
definable sets $\{T_i\}_{i \in I}$
such that for each $i \in I$,
all fibers $\A_{\z}$ with 
$\z \in T_i$ are definably homeomorphic. A very natural
question is to ask for an upper bound on the size of this  partition (which
will also give an upper bound on the number of homeomorphism types amongst the
arrangements $\A_{\z}, \z \in \R^{k_2}$).

Hardt's theorem is a corollary of
the existence of {\em cylindrical cell decompositions} of definable sets
proved in \cite{KPS} (see also \cite{Dries,Michel2}).
When $\A$ is a 
$(T,\pi_1,\pi_2)$-family for some fixed definable set 
$T \subset \R^{k_1+k_2 +\ell}$, with 
$\pi_1: \R^{k_1+k_2+\ell} \rightarrow \R^{k_1+k_2}$,
$\pi_2: \R^{k_1+ k_2+\ell} \rightarrow \R^{\ell}$, 
$\pi_2: \R^{k_1+ k_2} \rightarrow \R^{k_2}$
the usual projections, and $\#\A = n$,
the quantitative definable cylindrical  cell decomposition theorem in
\cite{Basu9} gives
a doubly exponential (in $k_1k_2$) upper bound on the cardinality of $I$ and
hence on the number of homeomorphism types  
amongst the arrangements  $\A_{\z}, \z \in \R^{k_2}$.
A tighter (say singly exponential) bound on the number of homeomorphism
types of the fibers would be very interesting but is unknown at present.
Note that we cannot hope for a bound which is better than singly exponential
because the lower bounds on the number of topological types proved in 
\cite{BV06} also applies in our situation.

In this paper we give tighter (singly exponential) upper bounds on the
number of homotopy types occurring amongst the fibers 
$\A_{\z}, \z \in \R^{k_2}$.
We prove the following theorems. The first theorem gives a bound on the
number of stable homotopy types 
of the arrangements $\A_{\z}, \z \in \R^{k_2}$, while the second
theorem gives a slightly worse bound for homotopy types.

\begin{theorem}
\label{the:homotopy_closed_stable}
There exists a constant $C = C(T) > 0$ such that 
for any collection
${\mathcal A} = \{A_1,\ldots,A_n\}$ of $(T,\pi_1,\pi_2)$-sets
the number of distinct  stable homotopy types amongst the arrangements 
$\A_{\z}, \z \in \R^{k_2}$ is bounded by
\[
C \cdot n^{(k_1+1)k_2}.
\]
\end{theorem}

If we replace stable homotopy type by homotopy type, we obtain 
a slightly weaker bound.

\begin{theorem}
\label{the:homotopy_closed_ordinary}
There exists a constant $C = C(T) > 0$ such that 
for any collection
${\mathcal A} = \{A_1,\ldots,A_n\}$ of $(T,\pi_1,\pi_2)$-sets
the number of distinct  homotopy types occuring amongst the
arrangements $\A_{\z}, \z \in \R^{k_2}$ is bounded by
\[
C \cdot n^{(k_1+3)k_2}.
\]
\end{theorem}

\section{Background}
In this section we describe some prior work in the area of bounding the
number of homotopy types of fibers of a definable map and their connections
with the results presented in this paper.

We begin with a definition.

\begin{definition}[$\A$-sets]
\label{def:A-sets}
Let ${\mathcal A} = \{A_1,\ldots,A_n\}$, such that each $A_i \subset \R^k$ 
is a $(T,\pi_1,\pi_2)$-set.
For $I \subset \{1,\ldots,n\}$, we let ${\mathcal A}(I)$ denote the set
\begin{equation}
\label{eqn:basic}
\bigcap_{i \in I \subset [n]} A_i  \;\; \cap 
\bigcap_{j \in [n]\setminus I}
(\R^k \setminus A_j) 
\end{equation}
and we will call such a set to be a basic ${\mathcal A}$-set.
We will denote by ${\mathcal C}({\mathcal A})$
the set of non-empty connected components of all basic
${\mathcal A}$-sets.
 
We will call definable subsets $S \subset \R^k$ defined by a
Boolean formula whose atoms are of the form,
$x \in A_i, 1 \leq i \leq n$, an ${\mathcal A}$-set. 
An ${\mathcal A}$-set is thus a union of basic ${\mathcal A}$-sets.
If  $T$ is closed, and 
the Boolean formula
defining $S$ has no negations, then $S$ is closed by definition
(since each $A_i$ is closed) and we call such a set an
${\mathcal A}$-closed set.

Moreover, if $V$ is any closed  definable subset of $\R^k$,
and $S$ is an ${\mathcal A}$-set (respectively ${\mathcal A}$-closed set), then we
will call $S \cap V$ to be an $({\mathcal A},V)$-set 
(respectively $({\mathcal A},V)$-closed set).
\end{definition}

\subsection{Bounds on the Betti numbers of Admissible Sets}
The problem of bounding the Betti numbers of ${\mathcal A}$-sets
is investigated in \cite{Basu9}, where several results known in the
semi-algebraic and semi-Pfaffian case are extended to this general
setting. In particular, we will need the following theorem proved there.

\begin{theorem}\cite{Basu9}
\label{the:betti}
Let ${\mathcal S}(\R)$ be an o-minimal structure over 
$\R$ and let $T \subset \R^{k+\ell}$
be a closed definable set. 
Then, there exists a constant $C = C(T) > 0$ depending only
on $T$ such that
for any arrangement ${\mathcal A} = \{A_1,\ldots,A_n\}$ 
of $(T,\pi_1,\pi_2)$-sets of $\R^k$ the following holds.

For every $i, 0 \leq i \leq k$,
\begin{equation}
\label{eqn:bettibound}
\sum_{D \in {\mathcal C}({\mathcal A})} b_i(D) \leq C \cdot n^{k-i}.
\end{equation}

\end{theorem}

\begin{remark}
The main intuition behind the bound in Theorem \ref{the:betti} 
(as well as similar results in the semi-algebraic 
and semi-Pfaffian settings) is
that the homotopy type (or at least the Betti numbers) 
of a definable set in $\R^k$  defined in
terms of $n$ sets belonging to some fixed definable family, 
depend only on the interaction of  these sets at most $k+1$ at a time.
This is reminiscent of Helly's theorem in convexity theory (see \cite{DGK63}) 
but in a homotopical setting. 
This observation is also used to give an efficient algorithm
for computing the Betti numbers of arrangements 
(see \cite[Section 8]{Basu_survey}).
However, the proof of Theorem \ref{the:betti} in 
\cite{Basu9} (as well as the proofs of similar results in the semi-algebraic
\cite{BPRbook2}
and semi-Pfaffian settings \cite{GV}) depends on an argument 
involving the Mayer-Vietoris sequence for homology, and does not
require more detailed information about homotopy types. 
In Section \ref{sec:homotopic_helly} below,
we make the above intuition mathematically precise. 

We prove two theorems (Theorems  \ref{the:union_stable} and 
\ref{the:union_ordinary} below) 
and these auxiliary results are the keys to 
proving the main results of this paper (Theorems
\ref{the:homotopy_closed_stable} and  \ref{the:homotopy_closed_ordinary}). 
Moreover, these auxiliary results could also be of independent interest
in the quantitative study of arrangements.
\end{remark}

\subsection
{Homotopy types of the fibers of a semi-algebraic map}
Theorem \ref{the:betti} gives tight bounds on the topological complexity
of an ${\mathcal A}$-set in terms of the cardinality of ${\mathcal A}$,
assuming that the sets in ${\mathcal A}$ belong to some fixed definable
family.
A problem closely related to the problem we consider in this paper is to
bound the number of topological types of the fibers of a projection restricted
to an arbitrary $\A$-set.

More precisely,
let $S \subset \R^{k_1+k_2}$ be a set definable in an o-minimal
structure over the reals (see \cite{Dries}) and let
$\pi: \R^{k_1+k_2} \rightarrow \R^{k_2}$ denote the projection
map on the last $k_2$ co-ordinates. We consider the fibers,
$S_{\z} = \pi^{-1}(\z) \cap S$ for different $\z$ in $\R^{k_2}$.
Hardt's trivialization theorem, 
(Theorem \ref{the:hardt} below) shows  that there exists a
definable partition of $\R^{k_2}$ into a finite number of 
definable sets $\{T_i\}_{i \in I}$
such that for each $i \in I$ and any point
$\z_i \in T_i$, $\pi^{-1}(T_i) \cap S$ is 
definably homeomorphic to
$S_{\z_i} \times T_i$ by a fiber preserving homeomorphism.
In particular, for each $i \in I$, all fibers $S_{\z}$ with 
$\z \in T_i$ are definably homeomorphic. 

In case $S$ is an ${\mathcal A}$-set, with $\A$ a 
$(T,\pi_1,\pi_2)$-family for some fixed definable set 
$T \subset \R^{k_1+k_2 +\ell}$, with 
$\pi_1: \R^{k_1+k_2+\ell} \rightarrow \R^{k_1+k_2}$,
$\pi_2: \R^{k_1+ k_2+\ell} \rightarrow \R^{\ell}$, 
$\pi_2: \R^{k_1+ k_2} \rightarrow \R^{k_2}$, 
the usual projections, and $\#\A = n$,
the quantitative definable  cylindrical cell decomposition theorem in
\cite{Basu9} gives
a doubly exponential (in $k_1k_2$) upper bound on the cardinality of $I$ and
hence on the number of homeomorphism types of the fibers of the map $\pi_3|_S$.
A tighter (say singly exponential) bound on the number of homeomorphism
types of the fibers would be very interesting but is unknown at present.

Recently, the problem of obtaining a tight bound on the number
of topological types of the fibers of a definable map for
semi-algebraic and semi-Pfaffian sets was 
considered in \cite{BV06}, and it was shown that the number of 
distinct homotopy types of the fibers of 
such a map  can be bounded (in terms of the format of the formula defining
the set) by a function  singly exponential in $k_1 k_2$. 
In particular, the combinatorial part of the bound is also singly exponential. 
A more precise statement in the case of semi-algebraic sets
is the following theorem  which appears in \cite{BV06}.

\begin{theorem}\cite{BV06}
\label{the:mainBV}
Let ${\mathcal P} \subset \R[X_1,\ldots,X_{k_1},Y_1,\ldots,Y_{k_2}]$,
with $\deg(P) \leq d$ for each $P \in {\mathcal P}$ and cardinality
$\#{\mathcal P} = n$.
Then, for any fixed ${\mathcal P}$-semi-algebraic set $S$
the number of different homotopy types of fibers $\pi^{-1} (\y) \cap S$
for various $\y \in \pi(S)$ is bounded by
\[
(2^{k_1} n k_2d)^{O(k_1 k_2)}.
\]
\end{theorem}

\begin{remark}
The proof of Theorem \ref{the:mainBV} however has the drawback that 
it relies on techniques involving perturbations of the original polynomials
in order to put them in general position, as well as
Thom's Isotopy Theorem, and
as such does not extend easily to the o-minimal setting.
The main results of this paper 
(see Theorem \ref{the:homotopy_closed_stable} 
and Theorem \ref{the:homotopy_closed_ordinary})
extend  the combinatorial part of 
Theorem \ref{the:mainBV} to the  more general o-minimal category.
\end{remark}

\begin{remark}
Even though the formulation of Theorem \ref{the:mainBV} seems a little
different from
the main theorems of this paper (Theorems 
\ref{the:homotopy_closed_stable} and \ref{the:homotopy_closed_ordinary}),
they are in fact closely related.
In fact, as a consequence of Theorem \ref{the:homotopy_closed_ordinary}
we obtain bounds on the number of homotopy types
of the fibers of $S$ for any fixed ${\mathcal A}$-set $S$, analogous to the
one in Theorem \ref{the:mainBV}.

More precisely we have:

\begin{theorem}
\label{the:homotopy_general_ordinary}
Let ${\mathcal S}(\R)$ be an o-minimal structure 
over $\R$,
and  
$T \subset \R^{k_1+k_2+\ell}$ a closed and bounded definable set, 
and $\pi_1: \R^{k_1+k_2+\ell}\rightarrow \R^{k_1 + k_2}$, 
$\pi_2: \R^{k_1+k_2+\ell}\rightarrow \R^{\ell}$, and 
$\pi_3: \R^{k_1 + k_2} \rightarrow \R^{k_2}$
the projection maps.
Then, there exists a constant $C = C(T) > 0,$ such that 
for any collection
${\mathcal A} = \{A_1,\ldots,A_n\}$ of $(T,\pi_1,\pi_2)$-sets,
for any fixed ${\mathcal A}$-set $S$
the number of distinct  homotopy types of 
fibers $\pi_3^{-1}(\z) \cap S$
for various $\z \in \pi_3(S)$ is bounded by
\[
C \cdot n^{(k_1+3)k_2}.
\]
\end{theorem}

A similar result with a bound of 
$C \cdot n^{(k_1+1)k_2}$ 
holds for
stable homotopy types as well.
\end{remark}

\section{A Topological Comparison Theorem}
\label{sec:homotopic_helly}
As noted previously, the main underlying idea behind our proof of
Theorem \ref{the:homotopy_closed_stable} is that
the homotopy type of an ${\mathcal A}$-set in $\R^k$ 
depends only on the interaction of sets in ${\mathcal A}$ 
at most $(k+1)$ at a time.
In this section we make this idea precise.

We show that in case 
$\A = \{A_1,\ldots,A_n\}$, with each $A_i$  
a definable, closed and bounded subset of $\R^k$, the  homotopy type of 
any ${\mathcal A}$-closed set is determined by 
a certain sub-complex of 
the {\em homotopy co-limit} of the
diagram  of $\A$.
The crucial fact here is that 
this sub-complex depends only on the intersections of
the sets in $\A$  at most $k+1$ at a time.

In order to avoid technical difficulties, we
restrict ourselves to the category of finite, regular cell complexes (see
\cite{Whitehead} for the definition of a regular cell complex). 
The setting of finite, regular cell complexes suffices for us,
since it is well known that
closed and bounded
definable sets in any o-minimal structure
are finitely triangulable, and hence, 
are homeomorphic to regular cell complexes.

\subsection{Topological Preliminaries}

Let ${\mathcal A} = \{A_1,\ldots,A_n\}$,  where each $A_i$ is a sub-complex
of a finite regular cell complex.
We now define the homotopy co-limit of 
the diagram of $\A$.

\subsubsection{Homotopy Co-limits}
Let $\Delta_{[n]}$ denote the standard simplex of dimension $n-1$ with
vertices in $[n]$ 
(and by $|\Delta_{[n]}|$ the corresponding closed geometric simplex).
For $I \subset [n]$, we denote by $\Delta_I$ 
the $(\#I-1)$-dimensional face of $\Delta_{[n]}$ corresponding
to $I$.

The homotopy co-limit, $\hocolimit(\A)$,  is a CW-complex defined as follows.
\begin{definition}[homotopy co-limit]
\label{def:hocolimit}
\[ 
\hocolimit(\A) =  
\coprod_{I \subset [n]} \Delta_I \times \A_I/\sim
\]
where the equivalence relation $\sim$ is defined as follows.

For $I \subset J \subset [n]$, let $s_{I,J}: |\Delta_I| 
\hookrightarrow |\Delta_J|$
denote the inclusion map of the face $|\Delta_I|$ in $|\Delta_J|$, and let
$i_{I,J}: |\A_J| \hookrightarrow |\A_I|$ denote the inclusion map of
$|\A_J|$ in $|\A_I|$.

Given $({\mathbf s},\x) \in |\Delta_I| \times |\A_I|$ and 
$({\mathbf t},\y) \in |\Delta_J| \times |\A_J|$ with $I \subset J$, 
then $({\mathbf s},\x) \sim 
({\mathbf t},\y)$ if and only if
${\mathbf t} = s_{I,J}({\mathbf s})$ and $\x = i_{I,J}(\y)$.
\end{definition}

Note that there exist two natural  maps 
$$
\displaylines{
f_\A: |\hocolimit(\A)| \rightarrow | \A^{[n]}|, \cr
g_\A: |\hocolimit(\A)| \rightarrow | \Delta_{[n]}|
}
$$
defined by
\begin{equation}
\label{eqn:f_A}
f_\A(\s,\x) = \s,
\end{equation}
and
\begin{equation}
\label{eqn:g_A}
g_\A(\s,\x) = \x.
\end{equation}
where
$(\s,\x) \in |{\Delta}_{I_c}| \times c $,
$c$ is a cell in  ${\mathcal A}^{[n]}$ and 
$I_c = \{i \in [n] \;\mid\; c \in  A_i \}$.

Notice that we have
\[
|\hocolimit(\A)| = 
\bigcup_{I \subset [n]} |\Delta_I| \times |\A_I| \subset
\bigcup_{I \subset [n]} |\Delta_I| \times \A^{[n]}. 
\]

\begin{definition}[truncated homotopy co-limits]
For any $m, 0 \leq m \leq n$, we will denote by  
$\hocolimit_m(\A)$ the  sub-complex of  $\hocolimit(\A)$ defined by
\begin{equation}
\label{eqn:truncatedhomotopycolimit}
\displaystyle{
\hocolimit_m(\A) = g_\A^{-1}({\rm sk}_m(\Delta_{[n]}))}.
\end{equation}
\end{definition}

\begin{definition}
[diagram preserving maps between homotopy co-limits]
\label{def:dgpheofhocolimits}
Replacing in Definition \ref{def:maps},  $| \A^{[n]}|$ and $|\B^{[n]}|$, by 
$|\hocolimit(\A)|$ and $|\hocolimit(\B)|$ respectively,
as well as  $|\A_I|$ and $|\B_I|$ 
by $f_\A^{-1}(|\A_I|)$ and $f_\B^{-1}(|\B_I|)$
respectively,
we get definitions of diagram preserving homotopy
equivalences and  stable homotopy equivalences between 
$|\hocolimit(\A)|$ and $|\hocolimit(B)|$, and more generally
for any $m \geq 0$, between $|\hocolimit_m(\A)|$ and 
$|\hocolimit_m(\B)|$.
\end{definition}

\begin{definition}
We say that $\A \approx_m \B$ if
there exists a diagram preserving homotopy equivalence
\[
\phi: |\hocolimit_m(\A)| \rightarrow |\hocolimit_m(\B)|.
\]

We say that $\A \sim_m \B$, if
there exists a diagram preserving stable homotopy equivalence
$\phi \in \{\hocolimit_m(\A);\hocolimit_m(\B)\}$, represented by
\[
\tilde{\phi}: 
\Suspension^N| \hocolimit_m(\A)|  \rightarrow \Suspension^N |\hocolimit_m(\B)|,
\]
for some $N > 0$.
\end{definition}

\begin{remark}
\label{rem:nomap}
Note that in the above definition the map $\phi$ need not be induced by a
diagram preserving map $\phi: \A^{[n]} \rightarrow \B^{[n]}$
(respectively, 
$
\tilde{\phi}: 
\Suspension^N| \hocolimit_m(\A)|  \rightarrow \Suspension^N |\hocolimit_m(\B)|
$).
Indeed if it was the case then the proofs of 
Theorems \ref{the:union_stable} and
\ref{the:union_ordinary} below would be simplified considerably.
\end{remark}

The two following theorems are the crucial topological ingredients in the
proofs of our main results.

\begin{theorem}
\label{the:union_stable}
Let ${\mathcal A} = \{A_1,\ldots,A_n\}, 
{\mathcal B} = \{B_1,\ldots,B_n\}$ be two families of sub-complexes   
of a finite regular cell complex, such that:
\begin{enumerate}
\item
$\HH_i(|\A^{[n]}|,\Z), \HH_i(|\B^{[n]}|,\Z) = 0$, for all $i \geq k$, and
\item
$\A \sim_k \B$.
\end{enumerate}

Then,
$\A$ and $\B$ are stable homotopy equivalent.
\end{theorem}

\begin{theorem}
\label{the:union_ordinary}
Let ${\mathcal A} = \{A_1,\ldots,A_n\}, 
{\mathcal B} = \{B_1,\ldots,B_n\}$ be two families of sub-complexes   
of a finite regular cell complex, such that:
\begin{enumerate}
\item
$\dim(A_i), \dim(B_i) \leq k$, for $1 \leq i \leq n$, and
\item
$\A \approx_{k+2} \B$.
\end{enumerate}
Then, 
$\A$ and $\B$ are homotopy equivalent.
\end{theorem}

We now state two corollaries  of Theorems \ref{the:union_stable}
and \ref{the:union_ordinary} which might be of interest.

Given a Boolean formula $\theta(T_1,\ldots,T_n)$ containing no negations
and a family of sub-complexes
${\mathcal A} = \{A_1,\ldots,A_n\}$ of a finite regular cell complex, 
we will denote by
${\mathcal A}_{\theta}$ the sub-complex defined by the formula,
$\theta_{\mathcal A}$, which is obtained from $\theta$ by replacing
in $\theta$ the atom $T_i$ by $A_i$ 
for each $i \in [n]$, and 
replacing each $\wedge$ (respectively $\vee$) by $\cap$ (respectively $\cup$).

\begin{corollary}
\label{cor:closed_stable}
Let ${\mathcal A} = \{A_1,\ldots,A_n\}, 
{\mathcal B} = \{B_1,\ldots,B_n\}$ be two families of sub-complexes   
of a finite regular cell complex, satisfying the same conditions
as in Theorem \ref{the:union_stable}.
Let $\theta(T_1,\ldots,T_n)$ be a Boolean formula without negations.
Then, 
$|\A_\theta|$ and $|\B_\theta|$ are stable homotopy equivalent.
\end{corollary}

\begin{corollary}
\label{cor:closed_ordinary}
Let ${\mathcal A} = \{A_1,\ldots,A_n\}, 
{\mathcal B} = \{B_1,\ldots,B_n\}$ be two families of sub-complexes   
of a finite regular cell complex, satisfying the same conditions
as in Theorem \ref{the:union_ordinary}.
Let $\theta(T_1,\ldots,T_n)$ be a Boolean formula without negations.
Then,
$|\A_\theta|$ and $|\B_\theta|$ are homotopy equivalent.
\end{corollary}

\subsection{Proofs of Theorems \ref{the:union_stable} and \ref{the:union_ordinary}}
\label{sec:proof}
Let $\A$ and $\B$ as in Theorem \ref{the:union_stable}.

We need a preliminary lemma.

\begin{lemma}
\label{lem:hom1}
\[
|{\mathcal A}^{[n]}| \;\;\mbox{is diagram preserving homotopy equivalent to}\;\; 
|\hocolimit(\A)|.
\]
\end{lemma}
\begin{proof}
Consider the map
\[
f_\A: |\hocolimit(\A)| \rightarrow | \A^{[n]}|
\]
defined in \eqref{eqn:f_A}.

Clearly, if $\x \in c$, $f_\A^{-1}(c) = |{\Delta}_{I_c}|$. Now applying 
Smale's version of the Vietoris-Begle Theorem \cite{Smale}
we obtain that $f_\A$ is a homotopy equivalence. Clearly, $f_\A$ is
diagram preserving. Moreover,
(see for instance the proof of Theorem 6 in \cite{Smale}) 
there exists an cellular inverse map
\[
h_\A: | \A^{[n]}| \rightarrow |\hocolimit(\A)|
\]  
such that $f_\A \circ h_\A$ is diagram preserving,
and is a homotopy inverse of $f_\A$.
\end{proof}

We can now prove Theorems \ref{the:union_stable} and
\ref{the:union_ordinary}.
\begin{proof}[Proof of Theorem  \ref{the:union_stable}]
Let
$h_{\A}: |{\A}^{[n]}| \rightarrow  |\hocolimit(\A)|$ 
be a diagram preserving 
homotopy equivalence known to exist by Lemma \ref{lem:hom1}. 
Since $h_\A$ is cellular, 
and $\dim |\A^{[n]}| \leq k $,
its image is contained in
$\hocolimit_k(\A)$ since 
by definition (Eqn. (\ref{eqn:truncatedhomotopycolimit}))
\[
{\rm sk}_k (\hocolimit(\A)) \subset \hocolimit_k(\A).
\]

We will denote by 
$h_{\A,\B}: \Suspension^N |\hocolimit_k (\A)| \rightarrow \Suspension^N
|\hocolimit_k(\B)|$ 
a map representing a
diagram preserving stable homotopy equivalence known to exist by hypothesis 
(which we assume to be cellular).

Let $i_{\B,k}: \Suspension^N |\hocolimit_k(\B)| \hookrightarrow \Suspension^N |\hocolimit(\B)|$ denote
the inclusion map. 
The map $i_{\B,k}$ induces isomorphisms
\[
(i_{\B,k})_*: \HH_j(\hocolimit_k(\B),\Z) \rightarrow \HH_j(\hocolimit(\B),\Z)
\]
for $0 \leq j \leq k-1$. 

Consequently, the map
$f_{\B}\circ i_{\B,k}$ induces isomorphisms
\[
(f_{\B}\circ i_{\B,k})_*: \HH_j(\hocolimit_k(\B),\Z) \rightarrow
\HH_j(\B^{[n]},\Z) 
\]
for $0 \leq j \leq  k-1$.

Composing the maps, $\Suspension^N h_{\A}, h_{\A\B}, i_{\B,k}, 
\Suspension^N f_{\B}$ we obtain that
the map,
\[
\Suspension^N f_{\B} \circ i_{\B,k} \circ h_{\A\B} \circ \Suspension^N h_{\A}:
\Suspension^N | \A^{[n]}| \rightarrow \Suspension^N |\B^{[n]}|
\]
induces isomorphisms 
\[
(\Suspension^N f_{\B} \circ i_{\B,k} \circ h_{\A,\B,k} \circ 
\Suspension^N h_{\A})_*:\HH_j( |\A^{[n]}|,\Z) \rightarrow \HH_j(|\B^{[n]}|,\Z)
\]
for all $j \geq 0$.

Moreover, the map 
$\Suspension^N f_{\B} \circ i_{\B,k} \circ h_{\A\B} \circ \Suspension^N h_{\A}$
is diagram preserving since each constituent of the composition is
diagram preserving.
It now follows from Theorem  \ref{the:stable}
that the 
S-map represented by
\[
\phi = \Suspension^N f_{\B} \circ i_{\B,k} \circ h_{\A\B} \circ \Suspension^N
h_{\A}:
\Suspension^N | \A^{[n]}| \rightarrow \Suspension^N |\B^{[n]}|,
\]
is a diagram preserving stable homotopy equivalence.
\end{proof}

Before proving Theorem \ref{the:union_ordinary} we first need to recall
a few basic facts from homotopy theory.

\begin{definition}[$k$-equivalence]
\label{def:k-equivalence}
A map $f: X \rightarrow Y$ between two regular cell complex is called a 
$k$-equivalence if the induced homomorphism 
\[
f_*: \pi_i(X) \rightarrow \pi_i(Y)
\]
is an isomorphism for all $0 \leq i < k$, and an epimorphism for $i=k$,
and we say that $X$ is $k$-equivalent to $Y$. (Note that $k$-equivalence
is not an equivalence relation).
\end{definition}

We also need the following well-known fact from algebraic topology.
\begin{proposition}
\label{prop:homotopy}
Let $X,Y$ be finite regular cell complexes with 
\[
\dim(X) < k, \dim(Y) \leq k,
\] 
and $f: X \rightarrow Y$ 
a $k$-equivalence. Then, $f$ is a homotopy equivalence between
$X$ and $Y$. 
\end{proposition}

\begin{proof}
See \cite[pp. 69]{Viro}.
\end{proof}

\begin{proof}[Proof of Theorem  \ref{the:union_ordinary}]
The proof is along the same lines as that of the proof of
Theorem \ref{the:union_stable}.
Let
$h_{\A}: |{\A}^{[n]}| \rightarrow |\hocolimit(\A)|$ 
be a diagram preserving 
homotopy equivalence known to exist by Lemma \ref{lem:hom1}. 
By the same argument as before, 
its image is contained in
$|\hocolimit_{k+2}(\A)|$.

We will denote by 
$h_{\A,\B}: |\hocolimit_{k+2} (\A)| \rightarrow |\hocolimit_{k+2}(\B)|$ 
a diagram preserving homotopy equivalence known to exist by hypothesis. 

Let $i_{\B,k+2}: |\hocolimit_{k+2}(\B)| \hookrightarrow |\hocolimit(\B)|$ denote
the inclusion map. 
The map $i_{\B,k+2}$ induces isomorphisms
\[
(i_{\B,k+2})_*: \pi_j(\hocolimit_{k+2}(\B)) \rightarrow \pi_j(\hocolimit(\B))
\]
for $0 \leq j \leq k+1$.
This is a consequence
of the exactness of the homotopy sequence of the pair 
$(\hocolimit(\B),\hocolimit_{k+2}(\B))$ (see \cite{Spanier}).

Consequently, the map
$f_{\B}\circ i_{\B,k}$ induces isomorphisms
\[
(g_{\B}\circ i_{\B,k})_*: \pi_j(\hocolimit_{k+2}(\B)) \rightarrow
\pi_j(\B^{[n]})
\]
for $0 \leq j \leq  k+1$.

Composing the maps, $h_{\A}, h_{\A\B}, i_{\B,k+2}, f_{\B}$ we obtain that
the map
\[
f_{\B} \circ i_{\B,k} \circ h_{\A\B} \circ h_{\A}:
| \A^{[n]}| \rightarrow |\B^{[n]}|
\]
induces isomorphisms 
\[
(f_{\B} \circ i_{\B,k} \circ h_{\A,\B,k} \circ h_{\A})_*:\pi_j(\A^{[n]}) \rightarrow \pi_j(\B^{[n]})
\] 
for $0 \leq j \leq k+1$.

Moreover, the map 
$f_{\B} \circ i_{\B,k} \circ h_{\A\B} \circ h_{\A}$
is diagram preserving since each constituent of the composition is
diagram preserving.
It now follows from Proposition  \ref{prop:homotopy}
that the map 
\[
\phi = f_{\B} \circ i_{\B,k} \circ h_{\A\B} \circ h_{\A}:
| \A^{[n]}| \rightarrow |\B^{[n]}|
\]
is a diagram preserving homotopy equivalence.
\end{proof}

\begin{proof}[Proof of Corollary  \ref{cor:closed_stable}]
First note that since the formula $\theta$ does not contain negations,
writing $\theta$ as a disjunction of conjunctions,  
there exists $\Sigma \subset 2^{[n]}$ such that
$
\displaystyle{
\A_\theta = \bigcup_{I \in \Sigma} \A_I
}
$
(respectively, 
$
\displaystyle{
\B_\theta = \bigcup_{I \in \Sigma} \B_I
}
$).
Let $\A' = \{\A_I \mid I \in \Sigma \}$ (respectively,
$\B' = \{\B_I \mid I \in \Sigma \}$).
It follows from the hypothesis that 
\[
\A' \sim_k \B'.
\]
Now apply Theorem \ref{the:union_stable}.
\end{proof}

\begin{proof}[Proof of Corollary  \ref{cor:closed_ordinary}]
The proof is similar to that of Corollary \ref{cor:closed_stable}
using Theorem \ref{the:union_ordinary} in place of 
Theorem \ref{the:union_stable} and is omitted.
\end{proof}

\section{Proofs of the Main Theorems}
\subsection{Summary of the main ideas}
We first  summarize the main ideas underlying the proof of
Theorem \ref{the:homotopy_closed_stable}. The proof of
Theorem \ref{the:homotopy_closed_ordinary} is similar
and differs only in technical details. 
Let 
$\A = \{A_1,\ldots,\A_n\}$ be a $(T,\pi_1,\pi_2)$- arrangement
in $\R^{k_1+k_2}$. Using Proposition \ref{prop:triangulation}, 
we obtain a definable
partition, $\{C_{\alpha}\}_{\alpha \in I}$ (say) of $\R^{k_2}$,
into connected locally closed definable sets $C_\alpha \subset \R^{k_2}$, 
with the property that as $\z$ varies over $C_\alpha$, 
we get for each $I \subset [n]$ with $\#I \leq k_1+1$ isomorphic (and
continuously varying) triangulations of the sub-arrangement $\A[I]$. 
Moreover, these triangulations are {\em downward compatible} in the sense that 
the restriction to 
$\A[J]$ of the triangulation of $\A[I]$,  refines that of $\A[J]$ for 
each $J \subset I$ (cf. Proposition \ref{prop:triangulation} below). 
These facts allow us to prove
that for any $z_1,z_2 \in C_\alpha$
the truncated homotopy co-limits $|\hocolimit_{k_1}(\A_{z_1})|$ and
$|\hocolimit_{k_1}(\A_{z_2})|$ are homotopy equivalent by a diagram preserving homotopy equivalence. 
More precisely, we first prove that the thickened homotopy co-limits
$|\hocolimit^+_{k_1}(\A_{z_1},\bar\eps)|$ and
$|\hocolimit^+_{k_1}(\A_{z_2},\bar\eps)|$ are homeomorphic, and then 
use Proposition \ref{prop:limit} to deduce that 
$|\hocolimit_{k_1}(\A_{z_1})|$ and
$|\hocolimit_{k_1}(\A_{z_2})|$ are homotopy equivalent.
Theorem \ref{the:union_stable} then implies that
$\A_{z_1}$ is stable homotopy equivalent to $\A_{\z_2}$ by a diagram
preserving stable homotopy equivalence.
It remains to bound the number of elements in the partition
$\{C_{\alpha}\}_{\alpha \in I}$. We use Theorem \ref{the:betti} 
to obtain a bound of  
$C \cdot n^{(k_1+1)k_2}$
on this number, where $C$ is a constant which  depends only on $T$.

In order to prove Theorem \ref{the:homotopy_closed_stable} 
we 
recall a few results from o-minimal geometry.

We first note an elementary property of families of admissible
sets (see \cite{Basu9} for a proof).

\begin{observation}
\label{obs:unionoffamilies}
Suppose that $T_1,\ldots,T_m \subset \R^{k+\ell}$ are 
definable sets, $\pi_1: \R^{k+\ell} \rightarrow \R^{k}$ and
$\pi_2: \R^{k+\ell} \rightarrow \R^{\ell}$ the two projections.
Then, there exists a definable 
subset $T' \subset \R^{k+\ell+m}$ 
depending only on $T_1,\ldots,T_m$, such that for any collection of 
$(T_i,\pi_1,\pi_2)$ families $\A_i$, $1 \leq i \leq m$,
the union
$
\displaystyle{
\bigcup_{i=1}^{m} \A_i
}
$ 
is a $(T',\pi_1',\pi_2')$-family, where 
$\pi_1': \R^{k+m+\ell} \rightarrow \R^{k}$ and
$\pi_2': \R^{k+\ell+m} \rightarrow \R^{\ell+m}$ are the 
projections onto the first $k$, and the last $\ell + m$ co-ordinates
respectively.
\end{observation}

\subsection{Hardt's Triviality for Definable Sets}
One important technical tool will be
the following o-minimal version of Hardt's triviality theorem.

Let $X \subset \R^k \times \R^\ell$ and $A \subset \R^k$ be
definable subsets  of $\R^k \times \R^\ell$ and $\R^\ell$ 
respectively,
and let $\pi: X \rightarrow \R^\ell$
denote the projection map on the last $\ell$ co-ordinates.

We say that {\em $X$ is definably trivial over $A$} 
if there exists a definable
set $F$ and a definable homeomorphism 
\[
h: F \times A  \rightarrow X \cap \pi^{-1}(A), 
\]
such that the following diagram commutes.
\[
\begin{diagram}
\node{F \times A}\arrow{e,t}{h}\arrow{s,t}{\pi_2}
\node{X \cap \pi^{-1}(A)}\arrow{sw,t}{\pi} \\
\node{A}
\end{diagram}
\]
In the diagram above  $\pi_2: F \times A \rightarrow A$ is the projection 
onto the second factor. We call $h$ 
{\em a definable trivialization of $X$ over $A$}.

If $Y$ is a definable subset of $X$, we say that the trivialization $h$ is
{\em compatible} with $Y$ if there is a definable subset $G$ of $F$ such
that $h(G \times A) = Y \cap \pi^{-1}(A)$. Clearly, the restriction of
$h$ to $G \times A$ is a trivialization of $Y$ over $A$.

\begin{theorem}[Hardt's theorem for  definable families]
\label{the:hardt}
Let $X \subset \R^k \times \R^\ell$ be a definable set and let
$Y_1,\ldots,Y_m$ be definable subsets of $X$. Then, there exists a 
finite partition of $\R^\ell$ into definable sets $C_1,\ldots,C_N$
such that $X$ is definably trivial over each $C_i$, and moreover
the trivializations over each $C_i$ are compatible with $Y_1,\ldots,Y_m$.
\end{theorem}

\begin{remark}
\label{rem:locally_closed}
We first remark that it is straightforward to derive from the proof of 
Theorem \ref{the:hardt} that the definable sets $C_1,\ldots,C_N$
can be chosen to be locally closed, and can be expressed as,
$C_1 = \R^\ell \setminus B_1, C_2 = B_1 \setminus B_2,
\ldots, C_N = B_{N-1} \setminus B_N$ for closed definable sets
$B_1,\ldots,B_N$. Clearly, the closed definable sets 
$B_1,\ldots,B_N$, determine the sets $C_i$ of the partition. 
\end{remark}

\begin{remark}
\label{rem:hardt}
Note also that it follows from Theorem \ref{the:hardt}, that
there are only a finite number of topological types amongst the
fibers of any definable map $f: X \rightarrow Y$ between definable
sets $X$ and $Y$. This remark would be used a number of times later
in the paper.
\end{remark}

Since in what follows we will need to consider many  different projections, 
we adopt the following convention.
\begin{notation}
\label{not:projections}
Given $m$ and $p$, $p\leq m$,  we will denote by 
\[
\pi_{m}^{\leq p}: \R^m \rightarrow \R^p
\]
(respectively $\pi_{m}^{>p}:\R^m \rightarrow \R^{m-p}$) 
the projection onto the first $p$ (respectively the last $m-p$) coordinates.
\end{notation}

\subsection{Definable Triangulations}

A triangulation
of a closed and bounded 
definable set $S$ is a simplicial complex 
$\Delta$ together with a
definable homeomorphism from $\vert \Delta\vert$ to $S$. 
Given such a triangulation we will often identify the simplices in
$\Delta$ with their images in $S$ under the given homeomorphism.
 
We call a triangulation $h_1: |\Delta_1| \rightarrow S$
of a definable set $S$, to be a {\em refinement}
of a triangulation
$h_2: |\Delta_2| \rightarrow S$ if
for every simplex $\sigma_1 \in \Delta_1$, there exists a simplex
$\sigma_2 \in \Delta_2$ such that $h_1(|\sigma_1|) \subset h_2(|\sigma_2|).$

Let $S_1 \subset S_2$ be two 
closed and bounded
definable subsets of
$\R^k$. We say that a definable 
triangulation $h: |\Delta| \rightarrow S_2$ of $S_2$, { \em respects}
$S_1$ if for every simplex $\sigma  \in \Delta$,
$h(\sigma) \cap S_1 = h(\sigma)$ or $\emptyset$.
In this case, $h^{-1}(S_1)$ is identified with a sub-complex
of $\Delta$ and
$h|_{h^{-1}(S_1)} :h^{-1}(S_1) \rightarrow S_1$ is a definable 
triangulation  of $S_1$. We will refer to this sub-complex  by 
$\Delta|_{S_1}$.

We introduce the following notational 
conventions 
in order to simplify arguments used later in the paper.

\begin{notation}
\label{not:unions}
If $T \subset \R^{k_1 + k_2 +\ell}$ be any definable subset
of $\R^{k_1+k_2+\ell}$,
for each $m \geq 0$, and $(\z,\y_0,\ldots,\y_m) \in \R^{k_2+ (m+1)\ell}$,
we will denote by $T_{\z,\y_0,\ldots,\y_m} \subset \R^{k_1}$ 
the definable set
$
\displaystyle{
\bigcup_{1 \leq i \leq m} \{\x \in \R^{k_1} \;\mid\;  (\x,\z) \in  T_{\y_i}\}
}
$.
For $\{j_0,\ldots,j_{m'}\} \subset [m]$, we will denote by 
$\pi_{m,j_0,\ldots,j_{m'}}: \R^{(m+1)\ell} \rightarrow \R^{(m'+1)\ell}$ the
projection map on the appropriate blocks of co-ordinates.
\end{notation}

It is well known that compact definable sets are triangulable and moreover
the usual proof of this fact (see for instance 
\cite{Michel2}) can be easily extended to 
produce a definable triangulation in a parametrized way. We will actually need 
a family of such triangulations satisfying certain compatibility
conditions mentioned before. The following proposition states the
existence of such families. We omit the proof of the proposition 
since it is a  technical but straightforward extension of the proof 
of existence of triangulations for definable sets.

\begin{proposition}[existence of $m$-adaptive triangulations]
\label{prop:triangulation}
Let $T \subset \R^{k_1 + k_2 +\ell}$ be a 
closed and bounded definable subset
of $\R^{k_1+k_2+\ell}$ and let $m \geq 0$.
For each $0 \leq p  \leq m$, there exists 
\begin{enumerate}
\item
a definable
partition $\{C_{p,\alpha}\}_{\alpha \in I_p}$ of $\R^{k_2 + (p+1)\ell}$, into
locally closed sets,  determined by a sequence of definable closed sets,
$\{B_{p,\alpha}\}_{\alpha \in I_p}$ (see Remark \ref{rem:locally_closed} above), 
and
\item
for each $\alpha \in I_p$, a  definable continuous map,
$$
\displaylines{
h_{p,\alpha}: |\Delta_{p,\alpha}| \times C_{p,\alpha} \rightarrow
\bigcup_{(\z,\y_0,\ldots,\y_p) \in C_{p,\alpha}} T_{\z,\y_0,\ldots,\y_p}
}
$$
where $\Delta_{p,\alpha}$ is a simplicial complex,
and such that for each $(\z,\y_0,\ldots,\y_p) \in 
C_{p,\alpha}$,
the restriction of
$h_{p,\alpha}$  to $|\Delta_{p,\alpha}| \times (\z,\y_0,\ldots,\y_p)$ is a 
definable triangulation
\[
h_{p,\alpha}: |\Delta_{p,\alpha}| \times (\z,\y_0,\ldots,\y_p) \rightarrow
T_{\z,\y_0,\ldots,\y_p} 
\]
of the definable set $T_{\z,\y_0,\ldots,\y_p}$ 
respecting the subsets,
$T_{\z,\y_0},\ldots,T_{\z,\y_p}$, and 
\item
for each subset $\{j_0,\ldots,j_{p'}\} \subset [p]$, 
$({\rm Id}_{k_2} ,\pi_{p,j_0,\ldots,j_{p'}})(C_{p,\alpha}) 
\subset C_{p',\beta}$
for some $\beta \in I_{p'}$, and
for each $(\z,\y_0,\ldots,\y_p) \in C_{p,\alpha}$, the definable triangulation
of $T_{\z,\y_{j_0},\ldots,\y_{j_{p'}}}$ induced by the triangulation
\[
h_{p,\alpha}: |\Delta_{p,\alpha}| \times (\z,\y_0,\ldots,\y_p) \rightarrow
T_{\z,\y_0,\ldots,\y_p}
\]
is a refinement of the definable triangulation,
\[
h_{p',\beta}: |\Delta_{p',\beta}| \times (\z,\y_{j_0},\ldots,\y_{j_{p'}}) 
\rightarrow
T_{\z,\y_{j_0},\ldots,\y_{j_{p'}}}.
\]
\end{enumerate}
(We will call the family $\{h_{p,\alpha}\}_{0 \leq p \leq m, \alpha \in I_p}$ 
an $m$-adaptive family of triangulations of $T$.)
\end{proposition}

We will also need the following technical result.
\begin{proposition}
\label{prop:limit}
Let $C_t \subset \R^k, t\geq 0$ be a definable family of 
closed and bounded sets,
and let $C \subset \R^{k+1}$ 
be the definable set 
$
\displaystyle{
\bigcup_{t \geq 0} C_t \times \{t\}.
}
$
If for every $0 \leq t < t'$,
$C_t \subset C_{t'}$,
and 
$
\displaystyle{
C_0  = \pi_{k+1}^{\leq k} (\overline{C} \cap 
(\pi_{k+1}^{>k})^{-1}(0)),
}
$
then there exists $t_0 > 0$ such that,
$C_0$ has the same homotopy type as $C_t$ for every $t$ with $0 \leq t \leq t_0$.
\end{proposition}

\begin{proof}
The proof given in \cite{BPRbook2} (see Lemma 16.17) for the semi-algebraic
case can be easily adapted to the o-minimal setting using Hardt's triviality
for definable families instead of for semi-algebraic ones.
\end{proof}

We now introduce another notational 
convention.

\begin{notation}
\label{not:infinitesimal}
Let ${\mathcal F}(x)$ be a predicate defined over $\R_+$ and $y \in \R_+$.
The notation $\forall(0 < x \ll y)\ {\mathcal F}(x)$
stands for the statement
$$\exists z \in (0,y)\ \forall x \in \R_+\ ({\rm if}\> x<z,\> {\rm then}\> 
{\mathcal F}(x)),$$
and can be read ``for all positive
$x$ sufficiently smaller than $y$, ${\mathcal F(x)}$ is true''.
\end{notation}

More generally, 
\begin{notation}
\label{not:small}
For 
$\bar \eps =(\eps_0,\ldots,\eps_n)$
and a predicate ${\mathcal F}(\bar \eps)$ over $\R_{+}^{n}$
we say
``for all sufficiently small $\bar \eps$, ${\mathcal F}(\bar \eps)$ is true'' 
if
$$
\forall(0 < \eps_{0} \ll 1)
\forall(0 < \eps_{1} \ll \eps_{0})
\cdots \forall(0 < \eps_{n} \ll \eps_{n-1})
{\mathcal F}(\bar \eps).
$$
\end{notation}

\subsection{Infinitesimal Thickenings of the Faces of a Simplex}
We will need the following construction.

Let $\bar\eps = (\eps_0,\ldots,\eps_n) \in \R_+^{n+1}$, with
$0 \leq  \eps_n <  \cdots < \eps_0 < 1$.
Later we will require $\bar\eps$ to be sufficiently small 
(see Notation  \ref{not:small}).

For a face $\Delta_J \in \Delta_{[n]}$, 
we denote by $C_{J}(\bar\eps)$ the subset of $|\Delta_J|$ defined by
\[
C_{J}(\bar\eps) = \{x \in |\Delta_J| \;\mid\; 
\dist(x,|\Delta_I|) \geq \eps_{\#I-1} \mbox{ for all }
I \subset J \}.
\]

Note that,
$$
\displaylines{
|\Delta_{[n]}| = 
\bigcup_{I \subset [n]} C_{I}(\bar\eps). 
}
$$

       \begin{figure}[hbt]
         \centerline{
           \scalebox{0.5}{
             \input{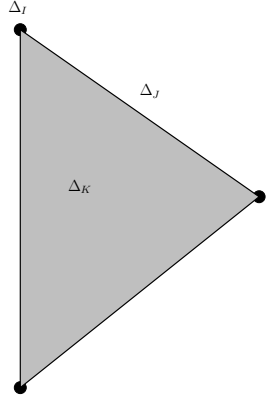}
             }
           }
         \caption{The complex $\Delta_{[n]}$.}
         \label{fig-eg1}
       \end{figure}

       \begin{figure}[hbt]
         \centerline{
           \scalebox{0.5}{
  		 \input{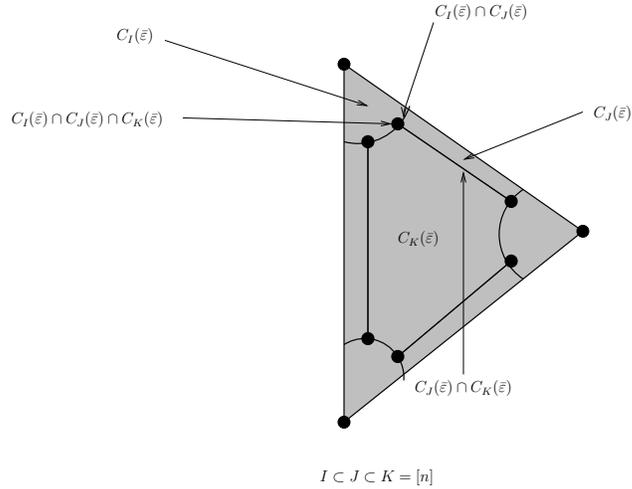}
             }
           }
         \caption{The corresponding complex 
${\mathcal C}(\Delta_{[n]})$ with $I \subset J \subset K =  [n]$.}
         \label{fig-eg2}
       \end{figure}

Also, observe that for sufficiently small $\bar\eps > 0$,
the various $C_J(\bar\eps)$'s 
are all homeomorphic to closed balls,
and moreover all non-empty intersections between them also have the 
same property.
Thus, the 
cells
$C_{J}(\bar\eps)$'s 
together with the
non-empty intersections between them form  a regular cell complex,
${\mathcal C}(\Delta_{[n]},\bar\eps)$, whose underlying
topological space is $|\Delta_{[n]}|$ 
(see Figures \ref{fig-eg1} and \ref{fig-eg2}). 

\begin{definition}
\label{def:defofC}
We will denote by ${\mathcal C}({\rm sk}_m(\Delta_{[n]}),\bar\eps)$
the sub-complex of ${\mathcal C}(\Delta_{[n]},\bar\eps)$ consisting of the
cells $C_{I}(\bar\eps)$'s 
together with the non-empty intersections between them 
where $|I| \leq m+1$.
\end{definition}

We now use thickened simplices defined above to define a thickened
version of the homotopy co-limit of an arrangement $\A$.

\subsection{Thickened Homotopy Co-limits}
Given an $m$-adaptive family of triangulations of $T$ (cf. Proposition
\ref{prop:triangulation}), 
$\{h_{p,\alpha}\}_{0 \leq p \leq m, \alpha \in I_p}$ 
and $\z \in \R^{k_2}$, 
we define a cell complex, 
$\hocolimit^+_m(\A_\z)$
(best thought of as an infinitesimally thickened version of
$\hocolimit_{m}(\A_\z)$),
whose associated topological space is homotopy equivalent
to $|\hocolimit_{m}(\A_\z)|$.

\begin{definition}[the cell complex $\hocolimit^+_m(\A_\z)$]
\label{def:defofK}
Let ${\mathcal C}_m$ denote the cell complex 
${\mathcal C}({\rm sk}_m(\Delta_{[n]}),\bar\eps)$
defined previously 
(cf.  Definition \ref{def:defofC}).

Let $C$ be a cell of ${\mathcal C}_m$.
Then, $C \subset |\Delta_I|$ for a unique simplex $\Delta_I$ 
with $I = \{i_0,\ldots,i_{m'}\} \subset [n]$, $m' \leq m$,
and (following notation introduced before in Definition \ref{def:defofC})

$$
\displaylines{
C =  C_{I_1}(\bar\eps) \cap \cdots \cap C_{I_p}(\bar\eps),
}
$$
with $I_1 \subset I_2 \subset \cdots \subset I_p \subset
I$ and $p \leq m'$. 

We denote by ${\mathcal K}(C,\bar\eps)$ the cell complex consisting of the
cells  
\[
C \times h_{m',\alpha}(|\sigma|,\z,\y_{i_0},\ldots,\y_{i_{m'}})
\]
with  $\alpha \in I_{m'}$, 
$(\z,\y_{i_0},\ldots,\y_{i_{m'}}) \in C_{\alpha,m'}$,
$\sigma \in \Delta_{m',\alpha}$, and
$
h_{m',\alpha}(|\sigma|,\z, \y_{i_0},\ldots,\y_{i_{m'}}) \subset
\A_{\z,I}
$.
We denote
\begin{equation}
\label{eqn:K_m(A_z)}
\hocolimit^+_m(\A_\z,\bar\eps)
= \bigcup_{C \in {\mathcal C}_m}
{\mathcal K}(C).
\end{equation}
\end{definition}

The compatibility properties 
(properties (2) and (3) in Proposition \ref{prop:triangulation})
of the $m$-adaptive family of triangulations of $T$,
$\{h_{p,\alpha}\}_{0 \leq p \leq m, \alpha \in I_p}$, 
ensure  that $\hocolimit^+_{m}(\A_{\z},\bar\eps)$ defined above is a regular
cell complex. Notice that, since the map $f_\A$ defined in Eqn. \ref{eqn:f_A} 
extends 
to  $|\hocolimit^+_m(\A_\z,\bar\eps)$, the notion of diagram
preserving maps extend to $|\hocolimit^+_m(\A_\z,\bar\eps)$ as well.

We now prove:
\begin{lemma}
\label{lem:thickenedlimit}
Let $\z \in \R^{\ell}$ and $m \geq 0$. Then, for all sufficiently small 
$\bar\eps > 0$,
$|\hocolimit^+_{m}(\A_{\z},\bar\eps)|$ 
is homotopy equivalent to $|\hocolimit_{m}(\A_{\z})|$
by a diagram preserving homotopy equivalence.
\end{lemma}

\begin{proof}
Let $N =  |\hocolimit^+_{m}(\A_{\z},\bar\eps)|$.
First replace 
$\eps_m$ 
by a variable $t$ in the definition of 
$N$
to obtain a closed and bounded definable set, 
$N_{t}^m$, 
and observe that 
$N_{t}^m \subset N_{t'}^m$ for all $0 < t < t' \ll 1$. 

Now apply 
Proposition \ref{prop:limit} to obtain that 
$N$ 
is homotopy equivalent
to 
$N_{0}^m$. 
Now, replace $\eps_{m-1}$ by $t$ in the definition of 
$N_{0}^m$ to obtain $N_{t}^{m-1}$, and applying Proposition
\ref{prop:limit} obtain that $N_{0}^m$ is homotopy equivalent to
$N_{0}^{m-1}$. Continuing in this way we finally obtain that,
$N$ is homotopy equivalent to $N_{0}^0 = |\hocolimit_{m}(\A_{z})|$.
Moreover, the diagram preserving property is clearly preserved
at each step of the proof.
\end{proof}

\begin{proof}[Proof of Theorem \ref{the:homotopy_closed_stable}]
Recall that for  $m \geq 0$, 
and $(\z,\y_0,\ldots,\y_m) \in \R^{k_2+ (m+1)\ell}$,
we denote by $T_{\z,\y_0,\ldots,\y_m}$ 
the definable set
$$
\displaylines
{
\bigcup_{i=1}^{m} T_{\z,\y_i} \subset \R^{\ell}.
}
$$

Now apply Proposition \ref{prop:triangulation} to the set $T$
with $m = k_1$ to  obtain an $k_1$-adaptive family of triangulations
$\{h_{p,\alpha}\}_{1 \leq p \leq k_1, \alpha \in I_p}$.

We now fix $\{\y_1,\ldots,\y_n \} \subset \R^{\ell}$ and let
$\A = \{A_1,\ldots,A_n\}$ with $A_i = T_{\y_i} \subset \R^{k_1+k_2}$.
For each $\z \in \R^{k_2}$, we will denote by
$\A_z = \{A_{1,\z},\ldots,A_{n,\z} \}$ where
$A_{i,\z} = \{\x \in \R^{k_1} \;\mid\; (\x,\z) \in A_i \}$.

For $\alpha \in I_{k_1}$, 
and $1 \leq i_0 < \cdots < i_{k_1} \leq n$,
we will denote by 
$B_{k_1,\alpha,{i_0},\ldots,{i_{k_1}}} \subset \R^\ell$ 
the definable closed set
$$
\displaylines{
B_{k_1,\alpha,{i_0},\ldots,{i_{k_1}}} = 
\{ \z \in \R^\ell \;\mid \; (\z,\y_0,\ldots,\y_{k_1}) \in B_{k_1,\alpha} \}.
}
$$ 

Let 
$$
\displaylines{
{\mathcal B} = 
\bigcup_{\alpha \in I_{k_1}}\{ B_{k_1,\alpha,{i_0},\ldots,{i_{k_1}}}\;\mid\;
1 \leq i_0 < i_1 < \cdots < i_{k_1} \leq n \},
}
$$
and let 
$C \in {\mathcal C}({\mathcal B})$.
Theorem \ref{the:homotopy_closed_stable} will follow from the 
following 
two lemmas.

\begin{lemma}
\label{lem:homotopy}
For any 
$\z_1,\z_2 \in C$, 
$\A_{\z_1}$ is 
stable homotopy equivalent to $\A_{\z_2}$.
\end{lemma}

\begin{proof}
Clearly, by Theorem \ref{the:union_stable} 
it suffices to prove that
$|\hocolimit_{k_1}(\A_{\z_1})|$ 
is diagram preserving 
homotopy equivalent to $|\hocolimit_{k_1}(\A_{\z_2})|$.

The compatibility properties of the triangulations 
ensure  that 
that the complex 
$|\hocolimit^+_{k_1}(\A_{\z_1},\bar\eps)$ is isomorphic to
$|\hocolimit^+_{k_1}(\A_{\z_2},\bar\eps)$ and hence
$|\hocolimit^+_{k_1}(\A_{\z_1},\bar\eps)|$ is homeomorphic to
$|\hocolimit^+_{k_1}(\A_{\z_1},\bar\eps)|$.

Using Lemma \ref{lem:thickenedlimit} 
we get a diagram preserving homotopy equivalence
$$
\displaylines{
\phi: |\hocolimit_{k_1}(\A_{\z_1})| \rightarrow 
|\hocolimit_{k_1}(\A_{\z_2})|.
}
$$ 

It now follows from Theorem \ref{the:union_stable} that 
the arrangements $\A_{\z_1}$ and $\A_{\z_2}$ are stable homotopy equivalent.
\end{proof}

\begin{lemma}
\label{lem:bound}
There exists a constant $C(T)$ such that the
cardinality of ${\mathcal C}({\mathcal B})$ is bounded by 
$C \cdot n^{(k_1+1)k_2}$.
\end{lemma}

\begin{proof}
Notice that each $B_{k_1,\alpha}, \alpha \in I_{k_1}$ is a definable
subset of $\R^{k_2 + (k_1+1)\ell}$ depending only on $T$. Also, the
cardinality of the index set $I_{k_1}$ is determined by $T$.

Hence, the set ${\mathcal B}$  consists of 
${n \choose {k_1+1}}$ definable sets,
each one of them is a 
\[
(B_{k_1,\alpha},\pi_{k_2+ (k_1+1)\ell}^{\leq k_2},
\pi_{k_2+ (k_1+1)\ell}^{> k_2})
\]
for some $\alpha \in I_{k_1}$.
Using Observation \ref{obs:unionoffamilies},
we have that ${\mathcal B}$ is a $(B,\pi_1',\pi_2')$-set for some
$B$ determined only by $T$. Now apply Theorem \ref{the:betti}.  
\end{proof}
The theorem now follows from Lemmas \ref{lem:homotopy} 
and \ref{lem:bound} proved above.  
\end{proof}

\begin{proof}[Proof of Theorem \ref{the:homotopy_closed_ordinary}]
The proof is similar to that of Theorem
\ref{the:homotopy_closed_stable} given above, except we use
Theorem \ref{the:union_ordinary} instead of Theorem \ref{the:union_stable},
and this accounts for the slight worsening of the exponent in the bound.
\end{proof}

\begin{proof}[Proof of Theorem \ref{the:homotopy_general_ordinary}]
Using a construction due to Gabrielov and Vorobjov
\cite{GV07} (see also \cite{Basu9})
it is possible to replace any given ${\mathcal A}$-set 
by  a closed bounded ${\mathcal A}'$-set
(where ${\mathcal A}'$ is a new family of definable closely
related to ${\mathcal A}$ with 
$\#\A' = 2k(\#\A)$),
such that the new set has the same homotopy type as the original one.
Using this construction
one can directly deduce Theorem \ref{the:homotopy_general_ordinary}
from Theorem \ref{the:homotopy_closed_ordinary}. We omit the details.
\end{proof}

\bibliographystyle{amsplain}
\bibliography{master}

\end{document}